\documentclass{amsart}
\usepackage{amsmath,amssymb}

\newtheorem{theorem}{Theorem}[section]

\newtheorem{lemma}[theorem]{Lemma}

\theoremstyle{definition}
\newtheorem{definition}[theorem]{Definition}

\theoremstyle{remark}
\newtheorem{remark}[theorem]{Remark}
\numberwithin{equation}{section}

\begin{document}

\title[Best Isoperimetric Constants for $(H^1, BMO)$-Normal Conformal Metrics]{Toward Best Isoperimetric Constants for $(H^1,BMO)$-Normal Conformal Metrics on $\mathbb R^n$, $n\ge 3$ \ $^\dag$}


\thanks{$^\dag$ Supported in part by Natural Science and
Engineering Research Council of Canada.}



\author{Jie Xiao}
\address{Department of Mathematics and Statistics, Memorial University of Newfoundland, St. John's, NL A1C 5S7, Canada}
\email{jxiao@math.mun.ca}


\date{}



\begin{abstract}
The aim of this article is: (a) To establish the existence of the
best isoperimetric constants for the $(H^1,BMO)$-normal conformal
metrics $e^{2u}|dx|^2$ on $\mathbb R^n$, $n\ge 3$, i.e., the
conformal metrics with the Q-curvature orientated conditions
$$
(-\Delta)^{n/2}u\in H^1(\mathbb R^n)\ \ \&\ \
u(x)=\hbox{const.}+\frac{\int_{\mathbb
R^n}(\log\frac{|\cdot|}{|x-\cdot|})(-\Delta)^{n/2}
u(\cdot)\,d\mathcal{H}^n(\cdot)}{2^{n-1}\pi^{n/2}\Gamma(n/2)};
$$
(b) To prove that $(n\omega_n^\frac1n)^\frac{n}{n-1}$ is the optimal
upper bound of the best isoperimetric constants for the complete
$(H^1,BMO)$-normal conformal metrics with nonnegative scalar
curvature; (c) To find the optimal upper bound of the best
isoperimetric constants via the quotients of two power integrals of
Green's functions for the $n$-Laplacian operators
$-\hbox{div}(|\nabla u|^{n-2}\nabla u)$.
\end{abstract}
\maketitle

\section{Introduction}

The original motivation of this paper goes back to one of the
geometric Q-curvature problems posed on Lawrence J. Peterson's
edited article -- Future Directions of Research in Geometry: A
Summary of the Panel Discussion at the 2007 Midwest Geometry
Conference (cf. \cite{Pet}).
\medskip

\noindent {\bf Alice Chang's Question}: {\it A very general question
is to ask ``What is the geometric content of Q-curvature?" For
example, we know that one can associate the scalar curvature with
the conformally invariant constant called the ``Yamabe constant".
When this constant is positive, it describes the best constant (in a
conformally invariant sense) of the Sobolev embedding of $W^{1,2}$
into $L^{2n/(n-2)}$ space; this in itself can be viewed as a
$W^{1,2}$ version of the isoperimetric inequality. It would be
interesting to know if Q-curvature, or the conformally invariant
quantity $\int Q$ associated with it, satisfies some similar
inequalities with geometric content.}
\medskip

To find out a way to attack this question let us choose a
conformally flat manifold $(\mathbb R^n, g)$ as the acting model --
the $2\le n$-dimensional Euclidean space $\mathbb R^n$ equipped with
the conformal metric $g=e^{2u}g_0$, where $u$ is a real-valued
smooth function on $\mathbb R^n$, i.e., $u\in C^\infty(\mathbb
R^n)$, and $g_0=|dx|^2=\sum_{k=1}^n dx_k^2$ is the standard
Euclidean metric on $\mathbb R^n$. For the convenience of statement
let us also agree to several more basic conventions. The symbols
$\Delta$ and $\nabla$ denote the Laplace operator
$\sum_{k=1}^n\partial^2/\partial x_k^2$ and the gradient vector
$(\partial/\partial x_1,...,\partial/\partial x_n)$ over $\mathbb
R^n$. The volume and surface area elements of the metric $g$ are
determined via
$$
dv_{g,n}=e^{nu}d\mathcal{H}^n\quad\hbox{and}\quad
ds_{g,n}=e^{(n-1)u}d\mathcal{H}^{n-1}
$$
where $\mathcal{H}^k$ stands for
the $k$-dimensional Hausdorff measure on $\mathbb R^n$. Thus, the
volume and surface area of the open ball $B_r(x)$ and its boundary
$\partial B_r(x)$ with radius $r>0$ and center $x\in\mathbb R^n$
take the following values:
$$
v_{g,n}\big(B_r(x)\big)=\int_{B_r(x)}e^{nu}\,d\mathcal{H}^n\quad\hbox{and}\quad
s_{g,n}\big(\partial B_r(x)\big)=\int_{\partial
B_r(x)}e^{(n-1)u}\,d\mathcal{H}^{n-1}.
$$
At the same time, on the conformally flat manifold $(\mathbb R^n,
g)$ there are two types of curvature -- one is the Ricci's scalar
curvature
$$
S_{g,n}=-2(n-1)e^{-2u}\Big(\Delta u+\frac{n-2}{2}|\nabla
u|^2\Big);
$$
and the other is the Paneitz's Q-curvature which, according as
\cite{FeGr} and \cite{Nd}, is given by
$$
Q_{g,n}=e^{-nu}(-\Delta)^{n/2} u.
$$
Here and hereafter, for $\alpha\in\mathbb R$ the operator
$(-\Delta)^{\alpha/2}$ is initially defined via the Fourier
transform
$$
\widehat{(-\Delta)^{\alpha/2}
f}(x)=(2\pi|x|)^\alpha\hat{f}(x)=(2\pi|x|)^\alpha\int_{\mathbb
R^n}e^{2\pi i x\cdot y}f(y)\,d\mathcal{H}^n(y),
$$
where $f$ is of the Schwartz class, denoted $f\in\mathcal S(\mathbb
R^n)$, that is,
$$
f\in C^\infty(\mathbb R^n)\quad\hbox{and}\quad
\sup_{x=(x_1,...,x_n)\in\mathbb
R^n}(1+|x|)^N\Big|\frac{\partial^{k_1+\cdots+k_n}f}{\partial^{k_1}
x_1\cdots\partial^{k_n}x_n}(x)\Big|<\infty
$$
for all multi-indices $(k_1,...,k_n)$ and natural numbers $N$. Of
course, the domain of $(-\Delta)^{\alpha/2}$ can be extended to
$C^\infty(\mathbb R^n)$ via the duality pairing:
$$
\langle (-\Delta)^{\alpha/2} f,h\rangle=\langle
f,(-\Delta)^{\alpha/2} h\rangle\quad\hbox{where}\quad f\in
C^\infty(\mathbb R^n)\quad\hbox{and}\quad h\in\mathcal{S}(\mathbb
R^n).
$$

In addition to the operators $S_{g,n}$ and $Q_{g,n}$, there is the
third operator related to the Laplacian, that is, the $n$-Laplacian
$$
\Delta_n u=-\hbox{div}\big(|\nabla u|^{n-2}\nabla u\big).
$$
Associated with this operator is the $n$-Green function
$G_\Omega(\cdot,\cdot)$ of a domain $\Omega\subset\mathbb R^n$ with
the boundary $\partial\Omega\not=\emptyset$, that is, the weak
solution to the Dirichlet problem:
\[
\left\{\begin{array} {r@{\quad,\quad}l}
\Delta_n G_{\Omega,n}(x,y)=\delta_y(x) & x\in\Omega\\
G_{\Omega,n}(x,y)=0 & x\in\partial\Omega.
\end{array}
\right.
\]
Here $\delta_y(x)$ is the Dirac measure. Of course, such a weak
solution does not always exist. Consequently, when a domain is
bounded and has the $n$-Green's function, the domain is said to be
bounded regular.

Since the scalar curvature $S_{g,2}/2$ and the Q-curvature $Q_{g,2}$
coincide with the classical Gaussian curvature $K$:
$$
\frac{S_{g,2}}{2}=Q_{g,2}=e^{-2u}(-\Delta)u=e^{-2u}\Delta_2 u=K
$$
which completely characterizes the curvature of the two-dimensional
conformally flat manifold $(\mathbb R^2, g)$, Chang's question leads
us to recall an easily-verified consequence of Li-Tam's
isoperimetric inequality (cf. \cite[Theorems 5.1-5.2 \& Corollary
5.3]{Li-Tam}), Finn's isoperimetric deficit formula \cite{Fin} and
Huber's isoperimetric inequality \cite[Theorem 3]{Hu}:
\medskip

\noindent {\bf Two-dimensional Theorem}: {\it For $u\in
C^\infty(\mathbb R^2)$ suppose $g=e^{2u}g_0$ is a conformal metric
on $\mathbb R^2$. Let
\begin{equation}\label{eq0}
\int_{\mathbb R^2}|Q_{g,2}|\,
dv_{g,2}<\infty\quad\hbox{and}\quad\int_{\mathbb
R^2}Q_{g,2}\,dv_{g,2}<2\pi.
\end{equation}
Then

\noindent{\rm(i)}
\begin{equation}\label{eq0a}
\kappa_{g,2}=\inf_{\Omega}\frac{\big(s_{g,2}(\partial\Omega)\big)^2}{v_{g,2}(\Omega)}=\inf_{f}\frac{\Big(\int_{\mathbb
R^2}|\nabla f|\,dv_{g,2}\Big)^2}{\int_{\mathbb R^2}|f|^2\,dv_{g,2}}
\end{equation}
is a positive number depending only on $(\mathbb R^2,g)$, where the
left-hand infimum is taken over all pre-compact domains
$\Omega\subseteq\mathbb R^2$ with $C^1$-boundary $\partial\Omega$,
and the right-hand infimum ranges over all $C^1$-functions $f$ with
compact support in $\mathbb R^2$.
\smallskip

\noindent{\rm(ii)}
\begin{equation}\label{eq0b}
\kappa_{g,2}=2\Big(2\pi-\int_{\mathbb R^2}Q_{g,2}\,dv_{g,2}\Big)
\end{equation}
holds for $Q_{g,2}\ge 0$, where $\kappa_{g,2}=4\pi$ if and only if
$g=g_0$. }
\medskip

Clearly, an appropriate higher-dimensional analogue of the
previously-quoted two-dimensional theorem (including condition
(\ref{eq0}) and assertions (i)-(ii)) would suggest a solution to
Chang's question for the Euclidean manifold $(\mathbb R^n,g)$. For
future use, the symbol $H^1(\mathbb R^n)$ (cf. \cite[Theorem
6.7.4]{Gr}) denotes the Hardy space of all real-valued functions $f$
on $\mathbb R^n$ that satisfy
$$
\|f\|_{H^1}=\int_{\mathbb R^n}|f|d\mathcal{H}^n+
\sum_{j=1}^n\int_{\mathbb R^n}|R_j(f)|d\mathcal{H}^n<\infty,
$$
where the Riesz transforms
$$
R_j(f)(x)=\lim_{\epsilon\to
0}\frac{\Gamma\big(\frac{n+1}{2}\big)}{\pi^\frac{n+1}{2}}\int_{|y|\ge\epsilon}
y_j|y|^{-n-1}f(x-y)\,d\mathcal{H}^n(y),\quad j=1,...,n
$$
are well-determined for $f\in L^1(\mathbb R^n)$ and the classical
gamma function $\Gamma(\cdot)$.

In addition, the best isoperimetric constant for a given conformal
metric $g$ on $\mathbb R^n$ is defined by
\begin{equation}\label{eq0k}
\kappa_{g,n}=\inf_{\Omega\in BDC(\mathbb
R^n)}\frac{\big(s_{g,n}\big(\partial
\Omega\big)\big)^{\frac{n}{n-1}}}{v_{g,n}\big(\Omega\big)},
\end{equation}
where $BDC(\mathbb R^n)$ represents the class of all bounded domains
$\Omega\subset\mathbb R^n$ with $C^1$-smooth boundary
$\partial\Omega$.

According to Chang's question as well as (\ref{eq0a}), our focus
should be on deciding when the sharp constant in (\ref{eq0k}) is
positive. Below is the outcome.

\begin{theorem}\label{t2} For $u\in C^\infty(\mathbb R^n)$ suppose $g=e^{2u}g_0$ is a conformal metric on $\mathbb R^n$, $n\ge
3$. If $g$ is $(H^1,BMO)$-normal, namely, if
\begin{equation}\label{eq11}
(-\Delta)^{n/2}u\in H^1(\mathbb R^n)
\end{equation}
and there is a constant $c$ such that
\begin{equation}\label{eq11c}
u(x)=c+\frac{\int_{\mathbb
R^n}\Big(\log\frac{|y|}{|x-y|}\Big)(-\Delta)^{n/2}
u(y)\,d\mathcal{H}^n(y)}{2^{n-1}\pi^{n/2}\Gamma(n/2)}\quad\hbox{for}\quad
x\in\mathbb R^n,
\end{equation}
then
\begin{equation}\label{eq33}
0<\kappa_{g,n}=\inf_{f\in C^1_0(\mathbb
R^n)}\frac{\big(\int_{\mathbb R^n}|\nabla
f|\,dv_{g,n}\big)^{\frac{n}{n-1}}}{\int_{\mathbb
R^n}|f|^{\frac{n}{n-1}}\,dv_{g,n}}<\infty,
\end{equation}
where the infimum ranges over $f\in C^1(\mathbb R^n)$ with compact
support in $\mathbb R^n$.
\end{theorem}

Perhaps it is worth pointing out that the notion of
$(H^1,BMO)$-normal is naturally inspired by both (\ref{eq11}) which
amounts to the following Q-curvature constraint:
$$
\int_{\mathbb R^n}\Big(|Q_{g,n}|+e^{-nu}\big|\nabla\big(
e^{(n-1)u}Q_{g,n-1}\big)\big|\Big)\,dv_{g,n}<\infty
$$
and the famous C. Fefferman's duality $[H^1(\mathbb
R^n)]^\ast=BMO(\mathbb R^n)$, John-Nirenberg's space of functions
with bounded mean oscillation in $\mathbb R^n$ (cf. \cite{Fe}),
which contains the function $\log|\cdot|/|x-\cdot|$ for any fixed
$x\in\mathbb R^n$. Here it is also worth mentioning that the
conditions
$$
\int_{\mathbb R^n}|Q_{g,n}|\,dv_{g,n}<\infty\quad\hbox{and}\quad
(\ref{eq11c})
$$
produce the definition for a conformal metric to be (classical)
normal -- see also \cite{Fin} for $n=2$; \cite[Definition
3.1]{ChQiYa} \& \cite[Definition 1.7]{ChQiYa1} for $n=4$; \cite{Fa}
\& \cite{BoHeSa1} for even integer $n\ge 4$; \cite{NdXi} \&
\cite{Xu} for any integer $n\ge 3$. Obviously, the
$(H^1,BMO)$-normal is stronger than the normal. From \cite{Hu0},
\cite{NdXi} and \cite{Xu} it turns out that any conformal metric $g$
on $\mathbb R^n$ with $n\ge 2$ satisfying
\begin{equation}\label{eq1111}
\int_{\mathbb R^n}|Q_{g,n}|\,dv_{g,n}<\infty\quad\hbox{and}\quad
\lim_{|x|\to\infty}\inf_{|y|>|x|}S_{g,n}(y)\ge 0
\end{equation}
is normal.

As a first application of Theorem \ref{t2}, we obtain the following
result (cf. (\ref{eq0b})) which seems most closely tied to Chang's
question above.

\begin{theorem}\label{t2a}  For $u\in C^\infty(\mathbb R^n)$ suppose $g=e^{2u}g_0$ is a complete conformal metric on $\mathbb R^n$, $n\ge 3$, with
\begin{equation}\label{eq111}
(-\Delta)^{n/2}u\in H^1(\mathbb R^n)\quad\hbox{and}\quad S_{g,n}\ge
0.
\end{equation}
Then
\begin{equation}\label{eq44}
0<\frac{\kappa_{g,n}}{(n\omega_n^\frac{1}{n})^\frac{n}{n-1}}\le
1-\frac{\int_{\mathbb
R^n}Q_{g,n}\,dv_{g,n}}{2^{n-1}\Gamma(n/2)\pi^{n/2}}=1,
\end{equation}
where
$$
\omega_n=\mathcal{H}^n\big(B_1(0)\big)={2\pi^{n/2}}\big(n\Gamma(n/2)\big)^{-1}
$$
is the $n$-dimensional Hausdroff measure of the unit ball $B_1(0)$
of $\mathbb R^n$. Moreover, the relation ``$\le$" in (\ref{eq44})
becomes the relation ``$=$" if and only if $g=g_0$.
\end{theorem}

As a second application of Theorem \ref{t2}, we gain the optimal
upper bound of $\kappa_{g,n}$ through a comparison between two
integrals of the Green function associated with the $n$-Laplacian
operator.

\begin{theorem}\label{t3} For $u\in C^\infty(\mathbb R^n)$ let $g=e^{2u}g_0$ be an $(H^1,BMO)$-normal conformal metric on $\mathbb R^n$, $n\ge 3$. Suppose
$BRD(\mathbb R^n)$ stands for the class of all bounded regular
domains $\Omega\subset\mathbb R^n$. Then

\item{\rm(i)}

\begin{equation}\label{eq55}
0<\frac{\kappa_{g,n}^{-q}\Gamma(q+1)}{\kappa_{g,n}^{-p}\Gamma(p+1)}\le\inf_{x\in\Omega\in
BRD(\mathbb
R^n)}\frac{\int_{\Omega}\big(G_{\Omega,n}(x,y)\big)^q\,dv_{g,n}(y)}{\int_{\Omega}\big(G_{\Omega,n}(x,y)\big)^p\,dv_{g,n}(y)}<\infty
\end{equation}
holds for $0\le q<p<\infty$. Moreover, the equality in (\ref{eq55})
is valid for $g=g_0$.

\item{\rm(ii)}

\begin{equation}\label{eq55a}
0<\frac{\kappa_{g,n}^{p+1}}{\Gamma(p+1)}\le\inf_{x\in\Omega\in
BRD(\mathbb
R^n)}\frac{\big(s_{g,n}(\partial\Omega)\big)^\frac{n}{n-1}}{\int_{\Omega}\big(G_{\Omega,n}(x,y)\big)^p\,dv_{g,n}(y)}<\infty
\end{equation}
holds for $0\le p<\infty$. Moreover, the equality in (\ref{eq55a})
holds for $g=g_0$.
\end{theorem}

The proofs of Theorems \ref{t2}-\ref{t2a}-\ref{t3} are provided in
the second, third and fourth sections respectively. Our techniques
and methods are of strong harmonic analysis flavor and developed
partially on the basis of the following works: \cite{BaBrFl},
\cite{BoHeSa}, \cite{ChQiYa}, \cite{DaSe}, \cite{Fa}, \cite{Nd},
\cite{NdXi}, and \cite{Wa}. Here we would like to thank P. Li for
sending us the motive paper \cite{Li-Tam}, A. Chang and G. Zhang for
reading the original version of this article, and the referee for
giving us helpful suggestions.

\section{Proof of Theorem \ref{t2}}

To prove Theorem \ref{t2}, we begin with the concept of
David-Semmes' strong $A_\infty$-weight (cf. \cite{DaSe}).

\begin{definition}\label{d21}
\item{\rm(i)} A function $w: \mathbb R^n\to [0,\infty)$ is called an $A_\infty$-weight provided there are constants $\epsilon>0$ and $C\ge 1$ such that
$$
\left(\big(\mathcal{H}^n(B)\big)^{-1}\int_B w^{1+\epsilon}\,d\mathcal{H}^n\right)^\frac1{1+\epsilon}\le C\big(\mathcal{H}^n(B)\big)^{-1}\int_B w\,d\mathcal{H}^n
$$
holds for all Euclidean balls $B\subset\mathbb R^n$.

\item{\rm(ii)} A nonnegative Borel measure $\mu$ on $\mathbb R^n$ is called a doubling measure provided there is a constant $C\ge 1$ such that $\mu(2B)\le C \mu(B)$ holds for every Euclidean ball $B=B_r(x)\subset\mathbb R^n$ and its doubling ball $2B=B_{2r}(x)$.

\item{\rm(iii)} A doubling measure $\mu$ on $\mathbb R^n$ is called a metric doubling measure provided there are a metric $d_\mu(\cdot,\cdot)$ on $\mathbb R^n$ and a constant $C\ge 1$ such that
$$
C^{-1}d_\mu(x,y)\le\mu\big(B_{|x-y|}(x)\cup B_{|y-x|}(y)\big)\le
Cd_\mu(x,y)\quad\hbox{for}\quad x,y\in\mathbb R^n.
$$
In this case, there exists an $A_\infty$-weight $w$ on $\mathbb R^n$ such that $d\mu=wd\mathcal{H}^n$ -- such a weight is said to be a strong $A_\infty$-weight.
\end{definition}

It is well-known that if $w$ is an $A_\infty$-weight then $u=\log
w\in BMO(\mathbb R^n)$:
$$
\|u\|_{BMO}=\sup_{B}\big(\mathcal{H}^n(B)\big)^{-1}\int_{B}\Big|u-\big(\mathcal{H}^n(B)\big)^{-1}\int_B
u\,d\mathcal{H}^n\Big|\,d\mathcal{H}^n<\infty,
$$
where the supremum is taken over all Euclidean balls
$B\subset\mathbb R^n$, and conversely, if $u\in BMO(\mathbb R^n)$
then there is a constant $c>0$ depending on $n$ and $\|u\|_{BMO}$
such that $w=e^{cu}$ is an $A_\infty$-weight. Moreover, a typical
example of the strong $A_\infty$-weight is the Jacobian determinant
$J_f$ of a quasiconformal mapping $f$ of $\mathbb R^n$ onto itself
in that if $d_\mu(x,y)=|f(x)-f(y)|$ then a change of variables plus
a distortion structure of quasiconformal mappings (cf.
\cite[p.380]{HeKiMa}) gives
$$
d_\mu(x,y)\approx \left(\mathcal{H}^n\Big(f\big(B_{|x-y|}(x)\cup B_{|y-x|}(y)\big)\Big)\right)^\frac1n\approx\left(\int_{B_{|x-y|}(x)\cup B_{|y-x|}(y)}J_f\,d\mathcal{H}^n\right)^\frac1n.
$$
Here and henceafter, $X\approx Y$ means $C^{-1}Y\le X\le CY$ for a
constant $C\ge 1$ independent of $X$ and $Y$, and moreover the
symbol $X\lesssim Y$ stands for $X\le CY$.

The lemma below is a straightforward consequence of David-Semmes'
\cite[(2.4)]{DaSe}.

\begin{lemma}\label{l21} If $w$ is a strong $A_\infty$-weight, then there is a constant $C>0$ such that the isoperimetric inequality
$$
\int_{\Omega}w\,d\mathcal{H}^n\le C\left(\int_{\partial\Omega}w^\frac{n-1}{n}\,d\mathcal{H}^{n-1}\right)^\frac{n}{n-1}
$$
holds for every bounded open set $\Omega\subset\mathbb R^n$.
\end{lemma}

From Bonk-Heinonen-Saksman's \cite[Theorem 3.1 \& Remark
3.26]{BoHeSa} we can readily obtain the following result.

\begin{lemma}\label{l22} Given $\alpha\in (0,n)$ and $x\in\mathbb
R^n$ let
$$
u(x)=(I_\alpha
f)(x)=\frac{\Gamma\big(\frac{n-\alpha}{2}\big)}{2^\alpha\pi^{\frac{n}{2}}\Gamma\big(\frac{\alpha}{2}\big)}\int_{\mathbb
R^n}\frac{f(y)}{|x-y|^{n-\alpha}}\,d\mathcal{H}^n(y)
$$
converge for some function $f:\mathbb R^n\to\mathbb R^1$ with
$$
\|f\|_{L^{n/\alpha}}=\left(\int_{\mathbb
R^n}|f|^{n/\alpha}\,d\mathcal{H}^n\right)^{\alpha/n}<\infty.
$$
Then $w=e^{nu}$ is a strong $A_\infty$-weight.
\end{lemma}

The forthcoming technical result is also useful.

\begin{lemma}\label{l23new} Let $0<\lambda<n$. Then
$$
\sup_{(r,x,y)\in(0,\infty)\times\mathbb R^n\times\mathbb
R^n}\frac{r^\lambda}{\mathcal{H}^{n}\big(B_{r}(x)\big)}\int_{B_{r}(x)}|z-y|^{-\lambda}\,d\mathcal{H}^{n}(z)<\infty.
$$
\end{lemma}
\begin{proof} Using a dyadic portion of $B_{r}(y)$ we estimate
\begin{eqnarray*}
&&\Big({\mathcal{H}^{n}\big(B_{r}(x)\big)}\Big)^{-1}\int_{B_{r}(x)}\frac{d\mathcal{H}^n(z)}{|z-y|^{\lambda}}
\\
&&\approx r^{-n}\left(\int_{B_{r}(x)\cap\big(\mathbb R^n\setminus B_r(y)\big)}\frac{d\mathcal{H}^{n}(z)}{|z-y|^\lambda}+\int_{B_{r}(x)\cap B_r(y)}\frac{d\mathcal{H}^{n}(z)}{|z-y|^\lambda}\right)\\
&&\lesssim
r^{-(n+\lambda)}\mathcal{H}^n\Big(B_{r}(x)\cap\big(\mathbb
R^n\setminus B_r(y)\big)\Big)\\
&&\quad+\ r^{-n}\sum_{k=0}^\infty\int_{B_{r}(x)\cap\big(B_{2^{-k}r}(y)\setminus B_{2^{-k-1}r}(y)\big)}\frac{d\mathcal{H}^{n}(z)}{|z-y|^\lambda}\\
&&\lesssim r^{-\lambda}\left(1+r^{-n}\sum_{k=0}^\infty
2^{k\lambda}\mathcal{H}^n\Big(B_{r}(x)\cap\big(B_{2^{-k}r}(y)\setminus
B_{2^{-k-1}r}(y)\big)\Big)\right)\\
&&\lesssim r^{-\lambda}\Big(1+\sum_{k=0}^\infty
2^{-k(n-\lambda)}\Big),
\end{eqnarray*}
whence getting the desired finiteness.
\end{proof}

\smallskip

\noindent{\bf Proof of Theorem \ref{t2}.}\ \ We first prove
$0<\kappa_{g,n}<\infty$. Using $(-\Delta)^{n/2}u\in H^1(\mathbb
R^n)$, the celebrated Stein-Weiss-Krantz's boundedness of $I_\alpha:
H^1(\mathbb R^n)\to L^{\frac{n}{n-\alpha}}(\mathbb R^n)$ (cf.
\cite{StWe} and \cite{Kr}), and $(-\Delta)^{-\frac12}=I_1$, we gain
\begin{eqnarray}\label{eq2ad}
&&\left(\int_{\mathbb
R^n}|(-\Delta)^{\frac{n-1}{2}}u|^\frac{n}{n-1}\,d\mathcal{H}^n\right)^\frac{n-1}{n}\nonumber\\
&&=\left(\int_{\mathbb
R^n}|I_1(-\Delta)^{\frac{n}{2}}u|^\frac{n}{n-1}\,d\mathcal{H}^n\right)^\frac{n-1}{n}\\
&&\lesssim\|(-\Delta)^{n/2}u\|_{H^1}\nonumber
\end{eqnarray}

Note also that for $n\ge 3$ and $x\not=y$ (cf. \cite[p.128, (2.10.1)
\& (2.10.8)]{KiSrTr} and \cite[p.132, (3)]{LieLo}),
\begin{eqnarray*}
&&(-\Delta)^\frac12\log|x-y|^{-1}\\
&&=(-\Delta)^{-\frac12}(-\Delta)\log|x-y|\\
&&=(n-2)I_1(|x-\cdot|^{-2})(y)\\
&&=\frac{(n-2)\Gamma\big(\frac{n-1}{2}\big)}{2\pi^{\frac{n}{2}}\Gamma\big(\frac{1}{2}\big)}\int_{\mathbb
R^n}{|x-z|^{-2}}{|z-y|^{1-n}}\,d\mathcal{H}^n(z)\\
&&=\left(\frac{\pi^\frac12\Gamma\big(\frac{n}{2}\big)}{\Gamma\big(\frac{n-1}{2}\big)}\right)|x-y|^{-1}.
\end{eqnarray*}
So if
\begin{equation*}\label{eq3ad}
u_1(x)=I_{n-1}\big((-\Delta)^{\frac{n-1}{2}}u\big)(x)\quad\hbox{for}\quad
x\in\mathbb R^n,
\end{equation*}
then
\begin{eqnarray*}
&&(-\Delta)^{\frac{n-1}{2}}u_1(x)\\
&&=\frac{\int_{\mathbb
R^n}\big((-\Delta)^{\frac{n-1}{2}}|x-y|^{-1}\big)(-\Delta)^{\frac{n-1}{2}}u(y)\,d\mathcal{H}^n(y)}{2^{n-1}\pi^{\frac{n-1}{2}}\Gamma\big(\frac{n-1}{2}\big)}\\
&&=\frac{\int_{\mathbb
R^n}\big((-\Delta)^{\frac{n-1}{2}}(-\Delta)^\frac12\log|x-y|^{-1}\big)(-\Delta)^{\frac{n-1}{2}}u(y)\,d\mathcal{H}^n(y)}
{2^{n-1}\pi^{\frac{n}2}\Gamma\big(\frac{n}{2}\big)}\\
&&=\int_{\mathbb
R^n}\delta_x(y)(-\Delta)^{\frac{n-1}{2}}u(y)\,d\mathcal{H}^n(y)\\
&&=(-\Delta)^{\frac{n-1}{2}}u(x)
\end{eqnarray*}
Here we have used the formula (cf. \cite[Proposition 2.1
(iv)]{NdXi}) that
$$
(-\Delta)^{n/2}(-\log|x-y|)=2^{n-1}\Gamma(n/2)\pi^{n/2}\delta_x(y)
$$
holds in the sense of distribution. Consequently,
$(-\Delta)^{\frac{n-1}{2}}(u-u_1)=0$. In other words,
$$
0=(2\pi|x|)^{n-1}\widehat{(u-u_1)}(x),\quad x\in\mathbb R^n.
$$
Since $n\ge 3$, this last equation forces $(-\Delta)(u-u_1)=0$,
namely, $u-u_1$ is a harmonic function on $\mathbb R^n$ and so is
each coordinate of the vector $\nabla(u-u_1)$.

A combined application of (\ref{eq11c}), the mean-value property of
$\partial(u-u_1)(y)/\partial y_j$, Fubini's theorem and Lemma
\ref{l23new} derives that for any $r>0$ and $x\in\mathbb R^n$,
\begin{eqnarray*}
&&\Big|\frac{\partial(u-u_1)}{\partial
y_j}(x)\Big|\\
&&=\left|\Big(\mathcal{H}^{n}\big(B_r(x)\big)\Big)^{-1}\int_{B_r(x)}\frac{\partial(u-u_1)}{\partial
y_j}(y)\,d\mathcal{H}^{n}(y)\right|\\
&&\lesssim\int_{\mathbb
R^n}\left(r^{-n}\int_{B_r(x)}\Big|\frac{\partial}{\partial
y_j}\log\frac{|z|}{|z-y|}\Big|\,
d\mathcal{H}^{n}(y)\right)|(-\Delta)^{n/2}u(z)|\,d\mathcal{H}^n(z)\\
&&\quad +\ r^{-n}\int_{B_r(x)}\left|\int_{\mathbb
R^n}\frac{\partial}{\partial
y_j}\Big(|z|^{-1}(-\Delta)^\frac{n-1}{2}u(y-z)\Big)\,d\mathcal{H}^n(z)\right|\,
d\mathcal{H}^{n}(y)\\
&&\lesssim\int_{\mathbb R^n}\left(r^{-n}\int_{B_r(x)}|z-y|^{-1}\,d\mathcal{H}^{n}(y)\right)|(-\Delta)^{n/2}u(z)|\,d\mathcal{H}^n(z)\\
&&\quad+\int_{\mathbb R^n}\left(r^{-n}\int_{B_r(x)}|z-y|^{-1}\,d\mathcal{H}^{n}(y)\right)|\nabla\big((-\Delta)^{(n-1)/2}u\big)(z)|\,d\mathcal{H}^n(z)\\
&&\lesssim
r^{-1}\Big(\|(-\Delta)^{n/2}u\|_{L^1}+\big\|\nabla\big((-\Delta)^{(n-1)/2}u\big)\big\|_{L^1}\Big)\\
&&\lesssim r^{-1}\|(-\Delta)^{n/2}u\|_{H^1},
\end{eqnarray*}
where we have also used the following formula (cf. \cite[p.58,
(1.94)]{MaZi}):
$$
-R_j(f)(x)=\frac{\partial}{\partial x_j}(I_1
f)(x)=\frac{\partial}{\partial x_j}\big((-\Delta)^{-1/2}
f\big)(x),\quad j=1,2,...,n.
$$
Letting $r\to\infty$ we obtain that $\nabla(u-u_1)$ is the zero
vector, whence finding that $u-u_1$ is a constant $c$. Now we get by
Lemma \ref{l22}, (\ref{eq2ad}) and the definition of $u_1$ that
$w=e^{nu}=e^{nc}e^{nu_1}$ is a strong $A_\infty$-weight. This,
together with Lemma \ref{l21}, deduces that for any $\Omega\in
BDC(\mathbb R^n)$,
$$
\int_{\Omega}e^{nu}\,d\mathcal{H}^n\le C\left(\int_{\partial\Omega}e^{(n-1)u}\,d\mathcal{H}^{n-1}\right)^\frac{n}{n-1}
$$
where $C>0$ is a constant independent of $\Omega$. Thus
$\kappa_{g,n}$ is a finite positive number.

Next, we prove
\begin{equation}\label{eq21bb}
\kappa_{g,n}=\inf_{f\in C^1_0(\mathbb R^n)}\frac{\big(\int_{\mathbb
R^n}|\nabla f|\,dv_{g,n}\big)^{\frac{n}{n-1}}}{\int_{\mathbb
R^n}|f|^{\frac{n}{n-1}}\,dv_{g,n}}.
\end{equation}
In spite of being well-known, such an argument is included here for
the completeness of the paper. For $t\ge 0$ and $f\in C^1_0(\mathbb
R^n)$, let
$$
\Omega(t;f)=\{x\in\mathbb R^n: |f(x)|\ge t\},
$$
then
$$
\partial\Omega(t;f)=\{x\in\mathbb R^n: |f(x)|=t\}.
$$
Thus, using the layer cake representation, the monotonicity of
$s_{g,n}\big(\partial\Omega(t;f)\big)$ with respect to $t\ge 0$ and
the co-area formula for $\nabla f$ (cf. \cite[Theorem
VIII.3.3]{Chav}) we obtain
\begin{eqnarray*}
&&\kappa_{g,n}\int_{\mathbb
R^n}|f|^\frac{n}{n-1}\,dv_{g,n}\\
&&=\kappa_{g,n}\int_0^\infty
v_{g,n}\big(\Omega(t;f)\big)\,dt^\frac{n}{n-1}\\
&&\le\int_0^\infty
\Big(s_{g,n}\big(\partial\Omega(t;f)\big)\Big)^\frac{n}{n-1}\,dt^\frac{n}{n-1}\\
&&=\Big(\frac{n}{n-1}\Big)\int_0^\infty
t^\frac{1}{n-1}\Big(s_{g,n}\big(\partial\Omega(t;f)\big)\Big)^\frac{n}{n-1}\,dt\\
&&\le\int_0^\infty\frac{d}{dt}\left(\Big(\int_0^t
s_{g,n}\big(\partial\Omega(r;f)\big)\,dr\Big)^\frac{n}{n-1}\right)\,dt\\
&&=\left(\int_0^\infty
s_{g,n}\big(\partial\Omega(t;f)\big)\,dt\right)^\frac{n}{n-1}\\
&&=\left(\int_{\mathbb R^n}|\nabla f|\,
dv_{g,n}\right)^\frac{n}{n-1},
\end{eqnarray*}
whence reaching
\begin{equation}\label{eq21bb1}
\kappa_{g,n}\le\inf_{f\in C^1_0(\mathbb
R^n)}\frac{\big(\int_{\mathbb R^n}|\nabla
f|\,dv_{g,n}\big)^{\frac{n}{n-1}}}{\int_{\mathbb
R^n}|f|^{\frac{n}{n-1}}\,dv_{g,n}}.
\end{equation}
To check the reversed inequality of (\ref{eq21bb1}), as to
$\Omega\in BDC(\mathbb R^n)$ and $\epsilon>0$ we choose the
following function
\[
f_\epsilon(x)=\left\{\begin{array} {r@{\quad,\quad}l}
1 & x\in\Omega\\
1-\epsilon^{-1}\hbox{dist}_g(x,\partial\Omega) & x\in\mathbb
R^n\setminus\Omega\ \ \& \ \ \hbox{dist}_g(x,\partial\Omega)<\epsilon\\
0 & x\in\mathbb R^n\setminus\Omega\ \ \& \ \
\hbox{dist}_g(x,\partial\Omega)\ge\epsilon.
\end{array}
\right.
\]
Here $\hbox{dist}_g(x,\partial\Omega)$ is the distance from $x$ to
$\partial\Omega$ with respect to the metric $g$. When $\epsilon$ is
small enough, we have that
\[
|\nabla f_\epsilon(x)|=\left\{\begin{array} {r@{\quad,\quad}l}
\epsilon^{-1} & x\in\mathbb R^n\setminus\overline{\Omega}\ \ \&
\ \ \hbox{dist}_g(x,\partial\Omega)<\epsilon\\
0 & \hbox{otherwise},
\end{array}
\right.
\]
where $\overline{\Omega}$ is the closure of $\Omega$, but also that
$f_\epsilon$ tends to the characteristic function $1_\Omega$ of
$\Omega$ as $\epsilon\to 0$. Hence
\begin{eqnarray*}
&&\lim_{\epsilon\to 0}\frac{\big(\int_{\mathbb R^n}|\nabla
f_\epsilon|\,dv_{g,n}\big)^{\frac{n}{n-1}}}{\int_{\mathbb
R^n}|f_\epsilon|^{\frac{n}{n-1}}\,dv_{g,n}}\\
&&=\frac{\Big(\lim_{\epsilon\to 0}
\epsilon^{-1}v_{g,n}\big(\{x\in\mathbb R^n\setminus\Omega:\
\hbox{dist}_g(x,\partial\Omega)<\epsilon\}\big)\Big)^\frac{n-1}{n}}{\lim_{\epsilon\to
0}\int_{\mathbb R^n}|f_\epsilon|^{\frac{n}{n-1}}\,dv_{g,n}}\\
&&=
\frac{\big(s_{g,n}(\partial\Omega)\big)^\frac{n}{n-1}}{v_{g,n}(\Omega)}
\end{eqnarray*}
and consequently,
\begin{equation}\label{eq21bb2}
\inf_{f\in C^1_0(\mathbb R^n)}\frac{\big(\int_{\mathbb R^n}|\nabla
f|\,dv_{g,n}\big)^{\frac{n}{n-1}}}{\int_{\mathbb
R^n}|f|^{\frac{n}{n-1}}\,dv_{g,n}}\le\kappa_{g,n}.
\end{equation}
Evidently, (\ref{eq21bb1}) and (\ref{eq21bb2}) imply (\ref{eq21bb}).

\begin{remark}\label{r1} (i) From \cite[Theorem 1.3]{BoHeSa1} and its odd-dimensional analog (cf. \cite{NdXi}) it follows that
there exists a dimensional constant $C_n\ge 1$ such that every
Euclidean manifold $(\mathbb R^n,g)$ with $n\ge 3$ is
$C_n$-biLipschitz equivalent to the background manifold $(\mathbb
R^n,g_0)$ -- in other words -- $e^{nu}$ is comparable to the
Jacobian determinant of a quasiconformal mapping from $\mathbb R^n$
to itself (this guarantees that $e^{nu}$ is a strong
$A^\infty$-weight), and hence (\ref{eq33}) holds, as along as $u\in
C^\infty(\mathbb R^n)$ satisfies (\ref{eq11c}) and
\begin{equation}\label{eq2f}
\int_{\mathbb
R^n}|(-\Delta)^{n/2}u|\,d\mathcal{H}^n<\frac{n2^{n-1}\Gamma(n/2)\pi^{n/2}}{
2^{7+4n}e^{4n(n-1)}3^{2n}}.
\end{equation}
Noticing the strict inclusion $H^1(\mathbb R^n)\subset L^1(\mathbb
R^n)$, we can immediately read off that the requirements
(\ref{eq11}) and (\ref{eq11c}) are a sufficient but not necessary
condition for (\ref{eq33}) to be true.

(ii) Under either the hypotheses of Theorem \ref{t2} or the
conditions (\ref{eq11c}) and (\ref{eq2f}), we can apply
\cite[Theorem]{DaSe} to establish the following inequality
concerning the best Sobolev constant for the conformal metric $g=e^{2u}g_0$:
\begin{equation*}\label{equp}
0<\inf_{f\in C^1_0(\mathbb R^n)}\frac{\big(\int_{\mathbb R^n}|\nabla
f|^p\,dv_{g,n}\big)^\frac1p}{\big(\int_{\mathbb
R^n}|f|^{\frac{pn}{n-p}}\,dv_{g,n}\big)^\frac{n-p}{pn}}<\infty\quad\hbox{where}\quad
1<p<n.
\end{equation*}

\end{remark}

\section{Proof of Theorem \ref{t2a}}

The forthcoming isoperimetric deficit formula (attached to the
Chern-Gauss-Bonnet integral inequality for $g=e^{2u}g_0$) is taken
from the main theorems in \cite{NdXi} and \cite{Xu}.

\begin{lemma}\label{l23} Let $u\in C^\infty(\mathbb R^n)$. If $g=e^{2u}g_0$ is complete conformal metric on $\mathbb R^n$, $n\ge 3$, but also satisfies
(\ref{eq1111}), then
$$
1-\frac{\int_{\mathbb
R^n}Q_{g,n}\,dv_g}{2^{n-1}\Gamma(n/2)\pi^{n/2}}=\lim_{r\to\infty}
\frac{\big(s_g\big(\partial B_r(0)\big)\big)^{n/(n-1)}}{(n\omega_n^{1/n})^{n/(n-1)}v_g\big(B_r(0)\big)}.
$$
\end{lemma}

\smallskip

\noindent{\bf Proof Theorem \ref{t2a}.}\ \ This follows from Lemma
\ref{l23}, Theorem \ref{t2}, the estimate
$$
\int_{\mathbb
R^n}|Q_{g,n}|\,dv_{g,n}=\|(-\Delta)^{n/2}u\|_{L^1}\le\|(-\Delta)^{n/2}u\|_{H^1},
$$
the vanishing integral condition
$$
\int_{\mathbb
R^n}(-\Delta)^{n/2}u\,d\mathcal{H}^n=0\quad\hbox{for}\quad
(-\Delta)^{n/2}u\in H^1(\mathbb R^n),
$$
and the evident inequality
$$
\inf_{\Omega\in BDC(\mathbb R^n)}\frac{\big(s_{g,n}\big(\partial
\Omega\big)\big)^{\frac{n}{n-1}}}{v_{g,n}\big(\Omega\big)}\le\lim_{r\to\infty}
\frac{\big(s_{g,n}\big(\partial
B_r(0)\big)\big)^{\frac{n}{n-1}}}{v_{g,n}\big(B_r(0)\big)}.
$$

Next, we handle the equality case of (\ref{eq44}). If $g=g_0$, then
$u=0$ which derives
$$
{\kappa_{g,n}}=\kappa_{g_0,n}={(n\omega_n^\frac{1}{n})^\frac{n}{n-1}}.
$$
Conversely, suppose
$\kappa_{g,n}=(n\omega_n^\frac{1}{n})^\frac{n}{n-1}$. Then
$$
\inf_{\Omega\in BDC(\mathbb R^n)}\frac{\big(s_{g,n}\big(\partial
\Omega\big)\big)^{\frac{n}{n-1}}}{v_{g,n}\big(\Omega\big)}={(n\omega_n^\frac{1}{n})^\frac{n}{n-1}}
$$
Now from the formula
$$
\frac{d\,v_{g,n}\big(B_r(x)\big)}{dr}=s_{g,n}\big(\partial
B_r(x)\big)\quad\hbox{for}\quad x\in\mathbb{R}^n\quad\hbox{and}\quad
r>0
$$
it follows that
$$
{(n\omega_n^\frac{1}{n})^\frac{n}{n-1}}\le\frac{\big(s_{g,n}\big(\partial
B_r(x)\big)\big)^{\frac{n}{n-1}}}{v_{g,n}\big(B_r(x)\big)}=\frac{\Big(\frac{d\,v_{g,n}\big(B_r(x)\big)}{dr}\Big)^{\frac{n}{n-1}}}
{v_{g,n}\big(B_r(x)\big)},
$$
namely,
$$
n\omega_n^\frac{1}{n}\le\Big(v_{g,n}\big(B_r(x)\big)\Big)^{\frac1n-1}
\frac{d\,v_{g,n}\big(B_r(x)\big)}{dr}.
$$
An integration acting on this last inequality gives
\begin{equation}\label{eq3a}
\omega_n r^n\le v_{g,n}\big(B_r(x)\big).
\end{equation}

On the other hand, the geometric interpretation of the scalar
curvature reveals (cf. \cite[3.98 Theorem]{GaHuLa})
$$
\frac{v_{g,n}\big(B_r(x)\big)}{\omega_n
r^n}=1-\frac{S_{g,n}(x)}{6(n+2)}r^2+o(r^2)\quad\hbox{as}\quad r\to
0.
$$
Since $S_{g,n}(x)\ge 0$ for $x\in\mathbb R^n$, we conclude
\begin{equation}\label{eq3b}
\lim_{r\to 0}\frac{v_{g,n}\big(B_r(x)\big)}{\omega_n r^n}\le 1.
\end{equation}
Using the previous estimates (\ref{eq3a})-(\ref{eq3b}) and the
fundamental theorem of Lebesgue (cf. \cite[pp.4-5]{Ste}), we find
$$
e^{nu(x)}=\lim_{r\to 0}(\omega_n
r^n)^{-1}\int_{B_r(x)}e^{nu}\,d\mathcal{H}^n=\lim_{r\to
0}\frac{v_{g,n}\big(B_r(x)\big)}{\omega_n r^n}=1\quad\hbox{for}\quad
x\in\mathbb R^n,
$$
whence getting $u=0$ and so $g=g_0$.
\medskip

\begin{remark}\label{r2} (i) Under the equality result of Theorem \ref{t2a}, the proof of \cite[Proposition
8.2]{He}, along with the extremal function
$$
f(x)=(1+|x|^\frac{p}{p-1})^\frac{p-n}{p}\quad\hbox{for}\quad
x\in\mathbb R^n,
$$
yields that for any $p\in (1,n)$ the well-known best Sobolev
constant
$$
\inf_{f\in C^1_0(\mathbb R^n)}\frac{\big(\int_{\mathbb R^n}|\nabla
f|^p\,dv_{g_0,n}\big)^\frac1p}{\big(\int_{\mathbb
R^n}|f|^{\frac{pn}{n-p}}\,dv_{g_0,n}\big)^\frac{n-p}{pn}}
$$
is equal to
$$
\left(\frac{n^\frac1{p-1}(n-p)}{p-1}\right)^{1-\frac1p}\left(\frac{\omega_n\Gamma\big(\frac{n}{p}\big)\Gamma\big(1+n-\frac{n}{p}\big)}{\Gamma(n)}\right)^\frac1n.
$$
It seems natural to conjecture that for any complete conformal
metric $g=e^{2u}g_0$ satisfying (\ref{eq111}), the inequality
$$
\inf_{f\in
C^1_0(\mathbb R^n)}\frac{\big(\int_{\mathbb R^n}|\nabla
f|^p\,dv_{g,n}\big)^\frac1p}{\big(\int_{\mathbb
R^n}|f|^{\frac{pn}{n-p}}\,dv_{g,n}\big)^\frac{n-p}{pn}}\le\left(\frac{n^\frac1{p-1}(n-p)}{p-1}\right)^{1-\frac1p}\left(\frac{\omega_n\Gamma\big(\frac{n}{p}\big)\Gamma\big(1+n-\frac{n}{p}\big)}{\Gamma(n)}\right)^\frac1n
$$
holds and the last equality happens when and only when $g=g_0$. Obviously, the last infimum is positive under the above-pointed suppositions.

(ii) Maybe it is appropriate to recall the so-called ``non-compact
Yamabe problem", which states: {\it On a smooth, complete,
non-compact $3\le n$-dimensional Riemannian manifold $(M,g)$, does
there exist a complete conformal metric of constant scalar
curvature?} Although this problem was answered negatively through Z.
Jin's counterexample in \cite{Jin}, it would still be of independent
interest to find a criterion for the $1$-scalar curvature equation
\begin{equation}\label{eqcur}
S_{g,n}=-2(n-1)e^{-2u}\Big(\Delta u+\frac{n-2}{2}|\nabla
u|^2\Big)=1
\end{equation}
to be solvable in a suitable function space. From Theorem \ref{t2a}
it is seen that if this equation has a solution $u$ belonging to
$C^\infty(\mathbb R^n)$ and obeying $(-\Delta)^{n/2}u\in H^1(\mathbb
R^n)$ then (\ref{eq44}) holds. A follow-up question arises: {\it Is
(\ref{eq44}) a sufficient condition for the existence of a solution
to (\ref{eqcur})?}
\end{remark}

\section{Proof of Theorem \ref{t3}}

To prove Theorem \ref{t3} let us review the so-called $C^1$ Sard
type theorem (cf. \cite[Theorem 10.4]{Sim})

\begin{lemma}\label{l30} Given a bounded domain $\Omega\subset\mathbb R^n$ with $n\ge 2$ let $f$ be a real-valued $C^1$ function on $\Omega$ with
$$
\sup_{x\in\Omega}\big(|f(x)|+|\nabla f(x)|\big)<\infty.
$$
Then
$$
f^{-1}(t)=\big(f^{-1}(t)\setminus\{x\in\Omega:\ \nabla f(x)=0\}\big)\cup\big(f^{-1}(t)\cap \{x\in\Omega:\ \nabla f(x)=0\}\big)
$$
holds for almost all $t\in f(\Omega)$, where
$f^{-1}(t)\setminus\{x\in\Omega:\ \nabla f(x)=0\}$ is an
$(n-1)$-dimensional $C^1$-submanifold with
$$
\mathcal{H}^{n-1}\big(f^{-1}(t)\cap\{x\in\Omega:\ \nabla
f(x)=0\}\big)=0\quad\hbox{and}\quad
\mathcal{H}^{n-1}\big(f^{-1}(t)\big)<\infty.
$$
Consequently, if $\mathsf{S}_f$ consists of the above $t$'s then $\mathcal{H}^1\big(f(\Omega)\setminus\mathsf{S}_f\big)=0$.
\end{lemma}

With the help of Lemma \ref{l30} and the asymptotic behavior of the
Green's function of $\Omega\in BRD(\mathbb R^n)$ below:
$$
G_{\Omega,n}(x,y)=-(n\omega_n)^\frac{1}{1-n}\log|x-y|+O(1)\quad\hbox{as}\quad
x\to y\quad\hbox{in}\quad\mathbb R^n,
$$
W. Wang discovered an integral formula for the $n$-Green
function (cf. \cite[Lemma 4.1]{Wa}) as follows.

\begin{lemma}\label{l31} Let $y\in\Omega\in BRD(\mathbb R^n)$ with $n\ge 2$. Then
$$
\int_{\{x\in\Omega:\ G_\Omega(x,y)=t\}}|\nabla G_\Omega(\cdot,y)|^{n-1}\,d\mathcal{H}^{n-1}(\cdot)=1
$$
holds for each $t\in\mathsf{S}_{G_\Omega(\cdot,y)}$.
\end{lemma}

\smallskip
\noindent{\bf Proof of Theorem \ref{t3}.}\ \ (i) For $t\ge 0$ and
$y\in\Omega\in BRD(\mathbb R^n)$ set
$$
\Omega(t,y;G)=\big\{x\in\Omega:\ G_{\Omega,n}(x,y)\ge t\big\}.
$$
Then $G_\Omega(\cdot,y)$ is of $C^1$ class on
$\Omega\setminus\{y\}$, and hence for
$t\in\mathsf{S}_{G_{\Omega,n}}$ we have
$$
\partial\Omega(t,y;G)=\{x\in\Omega:\ G_{\Omega,n}(x,y)=t\},
$$
which is the pre-image of $t$ under $G_{\Omega,n}(\cdot,y)$. From
now on, we will assume
$$
F(t,y)=v_{g,n}\big(\Omega(t,y;G)\big)=\int_{\Omega(t,y;G)}
e^{nu}\,d\mathcal{H}^n.
$$
On $\mathsf{S}_{G_{\Omega,n}}$ this function decreases -- in fact
$F(t,y)$ enjoys the differential equation (cf. \cite[p.53, Lemma
2.5]{Ban})
\begin{equation}\label{eq33a}
-\frac{dF(t,y)}{dt}=\int_{\partial\Omega(t,y;G)}\frac{e^{nu(x)}}{|\nabla
G_\Omega(x,y)|}\, d\mathcal{H}^{n-1}(x)\quad\hbox{for}\quad
t\in\mathsf{S}_{G_\Omega,n}.
\end{equation}
Applying H\"older's inequality, Lemma \ref{l31} and Theorem
\ref{t2}, we further derive from (\ref{eq33a}) that for
$t\in\mathsf{S}_{G_{\Omega,n}}$,
\begin{eqnarray*}
&&\Big(-\frac{dF(t,y)}{dt}\Big)^\frac{n-1}n\\
&&=\left(\int_{\partial\Omega(t,y;G)}\frac{e^{nu(x)}}{|\nabla
G_\Omega(x,y)|}\, d\mathcal{H}^{n-1}(x)\right)^\frac{n-1}n
\left(\int_{\partial\Omega(t,y;G)}\frac{d\mathcal{H}^{n-1}(x)}{|\nabla G_\Omega(x,y)|^{1-n}}\right)^\frac{1}n\\
&&\ge\int_{\partial\Omega(t,y;G)}{e^{(n-1)u(x)}}\, d\mathcal{H}^{n-1}(x)\\
&&\ge\kappa_{g,n}^\frac{n-1}{n}\Big(v_{g,n}\big(\Omega(t,y;G)\big)\Big)^\frac{n-1}{n}\\
&&=\kappa_{g,n}^\frac{n-1}{n}\big(F(t,y)\big)^\frac{n-1}{n}.
\end{eqnarray*}
The above inequalities yield
$$
\frac{d}{dt}\Big(e^{\kappa_{g,n}t}F(t,y)\Big)=e^{\kappa_{g,n}t}\left(\kappa_{g,n}F(t,y)+\frac{dF(t,y)}{dt}\right)\le
0.
$$
In other words, $e^{\kappa_{g,n}t}F(t,y)$ decreases with
$t\in\mathsf{S}_{G_{\Omega,n}}$.

Because Lemma \ref{l30} illustrates
$$
\mathcal{H}^1\big(\{t=G_{\Omega,n}(x,y)\in (0,\infty]:\
x\in\Omega\}\setminus\mathsf{S}_{G_{\Omega,n}}\big)=0,
$$
we can treat $F(\cdot,y)$ as a continuous and decreasing function on
$[0,\infty)$ but also $e^{\kappa_{g,n}t}F(t,y)$ as a decreasing
function with $t\in [0,\infty)$. Note that if $p>0$ and
$$
F_p(t,y)=\int_{\Omega(t,y;G)}\big(G_{\Omega,n}(x,y)\big)^p
e^{nu(x)}\, d\mathcal{H}^n(x),
$$
then
$$
F_p(0,y)=\int_{\Omega}\big(G_{\Omega,n}(x,y)\big)^p e^{nu(x)}\,
d\mathcal{H}^n(x)
$$
and hence, using the layer cake representation and integrating by
part, we deduce
$$
F_p(t,y)=-\int_t^\infty r^p\, dF(r,y).
$$
So, without loss of generality we may assume $F_q(0,y)<\infty$ for
$0\le q<p<\infty$ -- otherwise there is nothing to argue. Since
$d(e^{\kappa_{g,n}t}F(t,y))/dt\le 0$ , we conclude (via an
integration by part) that
$$
F_q(t,y)\le\kappa_{g,n}e^{\kappa_{g,n}t}\int_t^\infty r^q
e^{-\kappa_{g,n}r}\,dr
$$
and consequently,
$$
\frac{d}{dt}\log F_q(t,y)\le\frac{d}{dt}\log\int_t^\infty r^q
e^{-\kappa_{g,n}r}\,dr.
$$
Integrating this last differential inequality from $0$ to $t$, we
get
$$
\frac{F_q(t,y)}{F_q(0,y)}\le\frac{\kappa_{g,n}^{1+q}}{\Gamma(1+q)}\int_t^\infty
r^q e^{-\kappa_{g,n}r}dr.
$$
This estimate produces
\begin{eqnarray*}
&&F_p(0,y)\\
&&=-\int_0^\infty t^{p-q}t^q\,dF(t,y)\\
&&=(p-q)\int_0^\infty t^{p-q-1}F_q(t,y)\,dt\\
&&\le\frac{(p-q)\kappa_{g,n}^{q+1}F_{q}(0,y)}{\Gamma(q+1)}\int_0^\infty
t^{p-q-1}\left(\int_t^\infty r^qe^{-\kappa_{g,n}r}\,dr\right)\,dt\\
&&=\kappa_{g,n}^{q-p}\left(\frac{\Gamma(p+1)}{\Gamma(q+1)}\right)F_q(0,y),
\end{eqnarray*}
which in turn verifies (\ref{eq55}).

Furthermore, $g=g_0$ implies $u=0$ and
$\kappa_{g,n}=(n\omega_n^\frac1n)^\frac{n}{n-1}$. Now that
\begin{equation}\label{eq4b}
G_{B_1(0),n}(0,y)=-(n\omega_n)^{\frac1{1-n}}\log{|y|}\quad\hbox{for}\quad
y\in B_1(0),
\end{equation}
$g=g_0$ yields also

\begin{eqnarray*}
&&\frac{\kappa_{g,n}^{-q}\Gamma(q+1)}{\kappa_{g,n}^{-p}\Gamma(p+1)}\\
&&\le\inf_{x\in\Omega\in BRD(\mathbb
R^n)}\frac{\int_{\Omega}\big(G_{\Omega,n}(x,y)\big)^q\,dv_{g,n}(y)}{\int_{\Omega}\big(G_{\Omega,n}(x,y)\big)^p\,dv_{g,n}(y)}\\
&&\le\frac{\int_{B_1(0)}\big(G_{B_1(0),n}(0,y)\big)^q\,d\mathcal{H}^n(y)}{\int_{B_1(0)}\big(G_{B_1(0),n}(0,y)\big)^p\,d\mathcal{H}^n(y)}\\
&&=\frac{(n\omega_n)^{1-\frac{q}{n-1}}\int_0^1\Big(\log\frac1r\Big)^qr^{n-1}\,dr}{(n\omega_n)^{1-\frac{p}{n-1}}\int_0^1\Big(\log\frac1r\Big)^pr^{n-1}\,dr}\\
&&=\frac{(n\omega_n^\frac{1}{n})^{-\frac{qn}{n-1}}\Gamma(q+1)}{(n\omega_n^\frac{1}{n})^{-\frac{pn}{n-1}}\Gamma(p+1)}.
\end{eqnarray*}
Thus, the equality in (\ref{eq55}) occurs.

(ii) From Theorem \ref{t2} and the case $0=q<p<\infty$ of (i) it
follows that $\kappa_{g,n}>0$ and for any $x\in\Omega\in BRD(\mathbb
R^n)$,
$$
\frac{\kappa_{g,n}^p}{\Gamma(p+1)}\le\frac{v_{g,n}(\Omega)}{\int_{\Omega}\big(G_{\Omega,n}(x,\cdot)\big)^p\,dv_{g,n}(\cdot)}\le\frac{\kappa_{g,n}^{-1}\big(s_{g,n}(\partial\Omega)\big)^\frac{n}{n-1}}{\int_{\Omega}\big(G_{\Omega,n}(x,\cdot)\big)^p\,dv_{g,n}(\cdot)}.
$$
This derives (\ref{eq55a}). When $g=g_0$, as done in the last part of the foregoing (i) a calculation with (\ref{eq4b}) yields
$$
\frac{\kappa_{g_0,n}^{p+1}}{\Gamma(p+1)}=\frac{(n\omega_n^\frac1n)^\frac{n(p+1)}{n-1}}{\Gamma(p+1)}=\frac{\big(s_{g_0,n}(\partial
B_1(0))\big)^\frac{n}{n-1}}{\int_{B_1(0)}\big(G_{B_1(0),n}(0,\cdot)\big)^p\,dv_{g_0,n}(\cdot)},
$$
whence reaching the equality of (\ref{eq55a}).

\medskip

\begin{remark}\label{r3}{\rm(i)} We have not been able to prove
whether or not the equality of either (\ref{eq55}) or (\ref{eq55a})
implies $g=g_0$. Nevertheless we strongly conjecture that it has an
affirmative answer.

{\rm(ii)} When $w$ is the Jacobian determinant $J_f$ of a
quasiconformal map $f$ from $\mathbb R^n$ to itself, $w$ is a strong
$A_\infty$-weight and so by Lemma \ref{l21},
$$
\kappa_w=\inf_{\Omega\in BDC(\mathbb R^n)}\frac{\Big(\int_{\partial\Omega}w^\frac{n-1}{n}\, d\mathcal{H}^{n-1}\Big)^\frac{n}{n-1}}{\int_{\Omega}w\,d\mathcal{H}^n}>0.
$$
A careful look at the proof of Theorem \ref{t3} indicates that this theorem is still true with $\kappa_w$ replacing $\kappa_{g,n}$. In particular,
$$
\int_{\Omega}\big(G_{\Omega,n}(\cdot,y)\big)^p J_f(\cdot)\, d\mathcal{H}^n(\cdot)\le\frac{\Gamma(p+1)}{\kappa_w^p}\int_{\Omega}J_f\,d\mathcal{H}^n,
$$
where $y\in\Omega\in BDC(\mathbb R^n)$ and $0\le p<\infty$. This
observation suggests a future study of the quasiregular Q-space
${QRQ}_p(\Omega;\mathbb R^n)$ which comprises all quasiregular
mappings $f:\Omega\to\mathbb R^n$ with
$$
\sup_{y\in\Omega}\int_{\Omega}\big(G_{\Omega,n}(x,y)\big)^p |f'(x)|^n\,d\mathcal{H}^n(x)\approx\sup_{y\in\Omega}\int_{\Omega}\big(G_{\Omega,n}(x,y)\big)^p J_f(x)\,d\mathcal{H}^n(x)<\infty.
$$
Here $f'(x)$ means the formal derivative of $f$ at $x$, that is, the
matrix $[\partial f_j(x)/\partial x_k]_{n\times n}$ of the partial
derivatives $\partial f_j(x)/\partial x_k$, $j,k=1,...,n$, of the
coordinate functions $f_1,...,f_n$ of $f$. Moreover,
$|f'(x)|=\max_{h\in\partial B_1(0)}|f'(x)h|$. And, a continuous
mapping $f:\Omega\to\mathbb R^n$ is called quasiregular provided
that its coordinate functions $f_1,...,f_n$ lie in the local
homogeneous $n$-Sobolev space $\dot{W}^{1,n}_{loc}(\Omega)$, i.e.,
$$
\int_{O}|\nabla f_j|^n\,d\mathcal{H}^n<\infty,\quad j=1,...,n
$$
for each open set $O$ compactly contained in $\Omega$, and that
there exists a constant $\mathcal{K}\ge 1$ such that
\begin{equation}\label{eqquasi}
J_f(x)\le|f'(x)|^n\le \mathcal{K} J_f(x)
\end{equation}
is valid for almost all $x\in\Omega$. Especially, the quasiregular
homeomorphism is said to be a quasiconformal mapping. When $n=2$ and
$\mathcal{K}=1$ in (\ref{eqquasi}) the concept of
quasiregular/quasiconformal returns to the concept of
holomorphic/conformal. See also: \cite{HeKiMa} for more information
on the quasiregular mappings, \cite{XiJ1}-\cite{XiJ2} for an
overview of the recent research results on the holomorphic and
geometric $Q_p$-spaces on the unit disk of $\mathbb R^2$, and
\cite{La} for an investigation of the $Q_p$-type function space over
$B_1(0)$ introduced by a kind of invariance under M\"obius
transformations.
\end{remark}

\bibliographystyle{amsplain}

\end{document}